\documentclass[a4paper,12pt,reqno]{amsart}
\usepackage{amsfonts}
\usepackage{amsmath}
\usepackage{amssymb}
\usepackage[a4paper]{geometry}
\usepackage{mathrsfs}
\usepackage{csquotes}
\usepackage{mathrsfs}

\usepackage[colorlinks]{hyperref}
\renewcommand\eqref[1]{(\ref{#1})} 
\numberwithin{equation}{section}
\theoremstyle{plain}
\newtheorem{thm}{Theorem}[section]
\newtheorem{prop}[thm]{Proposition}

\theoremstyle{definition}

\def\e[#1]{{\textrm{e}}^{#1}}

\begin{document}

   \title[On $(\lambda,\mu)$-classes on the Engel group]
 {On $(\lambda,\mu)$-classes on the Engel group}

\author[M. Chatzakou]{Marianna Chatzakou}
\address{
  Marianna Chatzakou:
  \endgraf
  Department of Mathematics
  \endgraf
  Imperial College London
  \endgraf
  180 Queen's Gate, London SW7 2AZ
  \endgraf
  United Kingdom
  \endgraf
  {\it E-mail address} {\rm m.chatzakou16@imperial.ac.uk}
  }

     \begin{abstract}
     The purpose of this note is to compare the properties of the symbolic pseudo-differential calculus on the Heisenberg and on the Engel groups; nilpotent Lie groups of 2-step and 3-step, respectively. Here we provide a preliminary analysis of the structure and of the symbolic calculus with symbols parametrized by $(\lambda,\mu)$ on the Engel group,  while for the case of the Heisenberg group we recall the analogous results on the $\lambda$-classes of symbols. 
     \end{abstract}
     \maketitle

\section{Introduction}
\label{SEC:intro}
In \cite{FR16} the authors developed a global pseudo-differential calculus in the setting of a graded nilpotent Lie groups. Here we present the analogous preliminary results in the particular case of the Engel group $\mathcal{B}_4$.\\
\indent We prove that the representation of $\mathcal{B}_4$ is associated with the Kohn-Nirenberg quantization on $\mathbb{R}^4$. This, together with the analogue of the Kohn-Nirenberg quantization on Lie groups (c.f \cite{Tayl84},\cite{RT10},\cite{FR16}) gives rise to the development of the pseudo-differential calculus on $\mathcal{B}_4$ with scalar-valued symbols depending on the parameters $(\lambda,\mu)$-the co-adjoint orbits. \\
\indent In \cite{Tayl84}, M. Taylor describes a way one can develop a symbolic non-invariant calculus by defining a general quantization and the general symbols on any type-I Lie group, and explained his ideas in the setting of the Heisenberg group $\mathbb{H}_n$, with symbols defined by some asymptotic expansions. Particularising in the setting of a large class of nilpotent Lie groups; namely on the class of graded Lie groups, to the best of our knowledge, the development of a non-invariant calculus with scalar-valued symbols has been restricted to the case of the Heisenberg group (graded group of 2-step), see \cite{BFKG12}, or \cite{FR14}.\\
\indent Besides the amount of work devoted to the case of the Heisenberg group, the same motivating aspects appear as well on any graded Lie group. In our consideration of $\mathcal{B}_4$ (graded group of 3-step) our approach differs from the one in \cite{Tayl84} or in \cite{BFKG12} in the sense that the symbols are operator valued. However, using the link between the Kohn-Nirenberg quantization and the representations on $\mathcal{B}_4$ they can be expressed on the euclidean level. Concrete formulas for the difference operators in the setting of $\mathcal{B}_4$ are provided, laying down the necessary foundation for the characterisation of the symbol classes in our setting.

   \section{Prelimaries on the Engel group $\mathcal{B}_4$ and its Lie algebra}
   We start by fixing the notation required for presenting our results. The map 
   \begin{eqnarray*}
\lefteqn{(x_1,x_2,x_3,x_4) \circ(y_1,y_2,y_3,y_4)}\\
&:=& (x_1+y_1,x_2+y_2,x_3+y_3-x_1y_2,x_4+y_4+\frac{1}{2}x_1^2y_2-x_1y_3)
\end{eqnarray*}
endows $\mathbb{R}^4$ with a structure of a Lie group, and we shall refer to $\mathcal{B}_4=(\mathbb{R}^4,\circ)$ as the \textit{Engel group}. The Lie algebra of $\mathcal{B}_4$, say $\mathfrak{l}_4$, is, by the general theory, the vector space of (smooth) left invariant vector fields $X$ on $\mathbb{R}^4$ characterised by the property 
\[
(XI)(x)=\mathcal{J}_{\tau_x}(0) \cdot (XI)(0)\,,\forall x \in \mathcal{B}_4\,,
\] 
where $I$ is the identity map on $\mathbb{R}^4$ and $\mathcal{J}_{\tau_x}(0)$ denotes the Jacobian matrix at the origin of the map $\tau_x$ for $x \in \mathcal{B}_4$ where  $\tau_{x}(y):= x \circ y$ is the left translation by $x$ on $\mathcal{B}_4$. In particular we have
\[
\mathcal{J}_{\tau_x}(0)=\begin{pmatrix}
    1       & 0 & 0  & 0 \\
    0      & 1 & 0 & 0 \\
    0 & -x_1 & 1 & 0 \\
    0      & \frac{x_{1}^{2}}{2} & -x_1 & 1
\end{pmatrix}\,,
\]
so that for example for $X_1=\partial_{x_1}$, we can recognise that for every $x \in \mathcal{B}_4$,
\[
(X_1 I)(x)=\begin{pmatrix}
1 \\ 0 \\ 0 \\ 0
\end{pmatrix}=\begin{pmatrix}
    1       & 0 & 0  & 0 \\
    0      & 1 & 0 & 0 \\
    0 & -x_1 & 1 & 0 \\
    0      & \frac{x_{1}^{2}}{2} & -x_1 & 1
\end{pmatrix}\cdot \begin{pmatrix}
1 \\ 0 \\ 0 \\ 0
\end{pmatrix}=\mathcal{J}_{\tau_x}(0) \cdot (X_1 I)(0)\,,
\]
while similarly for the vector fields $X_2=\partial_{x_2}-x_1\partial_{x_3}+\frac{x_{1}^{2}}{2}\partial_{x_4}$, $X_3=\partial_{x_3}-x_1\partial_{x_4}$ and $X_4=\partial_{x_4}$. Simple calculations show that $[X_1,X_2]=X_3$, and $[X_1,X_3]=X_4$ \footnote{For smooth vectors $X$, $Y$ in $\mathbb{R}^n$, we define the Lie-bracket $[X,Y]:=YX-XY$.}, are the only non-zero relations, so that 
\[
\mathfrak{l}_4=span\{X_1,X_2, [X_1,X_2],[X_1,X_3] \}\,,
\]
and $X_1,X_2$ satisfy the so-called \textit{H\"{o}rmander condition}. The lower center series of $\mathfrak{l}_4$ defined inductively by 
\[
\mathfrak{l}_{4(1)}:=\mathfrak{l}_4\,,\quad \mathfrak{l}_{4(j)}=[\mathfrak{l}_4,\mathfrak{l}_{4(j-1)}]\footnote{For $V,W$ spaces of vector fields, we denote by $[V,W]$ the set $\{[v,w]:v \in V, w \in W\}$. }\,,
\]
terminates at $0$ after $3$ steps, that is $\mathcal{B}_4$ is of 3-step. In addition, $\mathcal{B}_4$ is a homogeneous Lie group on $\mathbb{R}^4$ since the mapping
\[
\delta_{\lambda}: \mathbb{R}^4 \rightarrow \mathbb{R}^4\,,\quad \delta_{\lambda}(x_1,x_2,x_3,x_4)=(\lambda x_1,\lambda x_2,\lambda^2 x_3,\lambda^3 x_4)\,,
\]
is an automorphism of $\mathcal{B}_4$ for every $\lambda>0$, and so the natural gradation of its Lie algebra $\mathfrak{l}_4$ appears as
\[
\mathfrak{l}_4=V_1\oplus V_2 \oplus V_3\,,
\]
where $V_1=span\{X_1,X_2\}\,,V_2=span\{X_3\}$ and $V_3=span\{X_4\}$, and is such that $[V_i,V_j] \subset V_{i+j}$, $i \neq j$. \\

Finally, let us note that, from the general theory of homogeneous Lie groups, the Lebesgue measure on $\mathbb{R}^4$ is invariant with respect to the left and right invariant translation on $\mathcal{B}_4$, that is the Lebesgue measure on $\mathbb{R}^4$ is the Haar measure for $\mathcal{B}_4$ and we can formulate as 
\[
\int_{\mathcal{B}_4}\cdots dx_1 dx_2 dx_3 dx_4=\int_{\mathbb{R}^4}\cdots dx_1 dx_2 dx_3 dx_4\,.
\]
   \section{Group representation and quantization of the Fourier transform}
  
   The representations of the Engel group $\mathcal{B}_4$ are the infinite dimensional unitary (equivalence classes of) representations of $\mathcal{B}_4$. Parametrised by $\lambda \neq 0$ and $\mu \in \mathbb{R}$, following \cite[p.333]{Dix57}, they act on $L^2(\mathbb{R}^n)$. We denote them by $\pi_{\lambda,\mu}$, and realise them as 
   \[
   \pi_{\lambda,\mu}(x_1,x_2,x_3,x_4)h(u)= exp \left( i \left(-\frac{\mu}{2 \lambda}x_2+\lambda x_4-\lambda x_3 u+\frac{\lambda}{2}x_2 u^2 \right) \right)h(u+x_1)\,,
   \] 
   for $h \in L^2(\mathbb{R}),u \in \mathbb{R}$. The group Fourier transform of a function $f \in L^1(\mathcal{B}_4)$ is by definition the linear endomorphism on $L^2(\mathbb{R})$
   \[
   \mathcal{F}_{\mathcal{B}_4}(f)(\pi_{\lambda,\mu}) \equiv \hat{f}(\pi_{\lambda,\mu})\equiv \pi_{\lambda,\mu}(f):= \int_{\mathcal{B}_4}f(x)\pi_{\lambda,\mu}(x)^{*}\,dx\,.
   \]
   Rigorous computations show that  $\hat{f}(\pi_{\lambda,\mu})h(u)$ can be written as 
\begin{align}
\label{fourier,tansform}
\lefteqn{ \int_{\mathbb{R}^4} \bigg[ f(x_1,x_2,x_3,x_4)}\nonumber\\
& \cdot \exp \left( i \left( \frac{\mu}{2 \lambda}x_2 - \lambda x_4 +\lambda x_3 (u-x_1)-\frac{\lambda}{2}x_2 (u-x_1)^{2} \right) \right) h(u-x_1)\bigg] dx_1\,dx_2\,dx_3\,dx_4\tag{1}\\
& = (2 \pi)^{-2}\int_{\mathbb{R}^4} \int_{\mathbb{R}^4} \bigg[ \mathcal{F}_{\mathbb{R}^4}(f) (\xi, \eta, \tau, \omega) \cdot e^{ix_1 \xi}\cdot e^{i x_2 \eta}\cdot e^{i x_3 \tau}\cdot e^{i x_4 \omega} \nonumber\\
& \cdot \exp \left( i \left( \frac{\mu}{2 \lambda}x_2 - \lambda x_4 +\lambda x_3 (u-x_1)-\frac{\lambda}{2}x_2 (u-x_1)^{2} \right) \right) \nonumber\\ & \cdot h(u-x_1)\bigg] dx_1\,dx_2\,dx_3\,dx_4 d\xi\,d\eta\,d\tau\,d\omega \nonumber \\
&= -(2\pi) \int_{\mathbb{R}}\int_{\mathbb{R}} \bigg[ e^{i x_1 \xi} \mathcal{F}(f)_{\mathbb{R}^4} ( \xi, \frac{\lambda}{2}(u-x_1)^2-\frac{\mu}{2 \lambda}, \lambda(x_1-u), \lambda)h(u-x_1) \bigg] dx_1\,d\xi\nonumber\\
&=(2 \pi) \int_{\mathbb{R}} \int_{\mathbb{R}}\bigg[ e^{i(u-v)\xi} \mathcal{F}(f)_{\mathbb{R}^4}(\xi,  \frac{\lambda}{2}v^2-\frac{\mu}{2 \lambda}, -\lambda v, \lambda)h(v)\bigg] dv\,d\xi\nonumber\,,
\end{align}    
   for $h \in L^2(\mathbb{R})$ and $u \in \mathbb{R}$, that is
   \begin{equation}\tag{2}\label{g.f.t}
   \mathcal{F}_{\mathcal{B}_4}(f)(\pi_{\lambda,\mu})=Op[a_{f,\lambda,\mu}(\cdot,\cdot)]\,,
   \end{equation}
   where \[a_{f,\lambda,\mu}(v,\xi)=(2\pi)^2\mathcal{F}_{\mathbb{R}^4}(f)(\xi,\frac{\lambda}{2}v^2-\frac{\mu}{2\lambda},-\lambda v,\lambda)\,.\]
  
 Here the Fourier transform $\mathcal{F}_{\mathbb{R}^4}$ is defined via:
   \begin{equation}
   \mathcal{F}_{\mathbb{R}^4}f(\xi)=(2\pi)^{-2}\int_{\mathbb{R}^4} f(x)e^{-ix\xi}\,dx \quad (\xi \in \mathbb{R}^4, f \in L^{1}(\mathbb{R}^4))\,,
   \end{equation}
   and $Op$ denotes the Kohn-Nirenberg quantization, that is for a smooth symbol $a$ on $\mathbb{R} \times \mathbb{R}$ the operator 
   \[
   Op(a)f(u)=(2\pi)^{-1}\int_{\mathbb{R}}\int_{\mathbb{R}} e^{i(u-v)\xi}a(v,\xi)f(v)\,dv\,d\xi\,,
   \]
   for $f \in \mathcal{S}(\mathbb{R})$ and $u \in \mathbb{R}$. \\
   
We note that for the case of the Heisenberg group $\mathbb{H}_n$ the group Fourier transform has been computed in \cite{FR14} as being the operator 
\begin{equation}\tag{3}\label{gfthei}
\mathcal{F}_{\mathbb{H}_n}(f)(\pi_{\lambda}) =(2\pi)^{\frac{n}{2}}Op^{W}[\mathcal{F}_{\mathbb{R}^{2n+1}}(f)(\sqrt{|\lambda|}\cdot,\sqrt{\lambda}\cdot,\lambda)]\,,
\end{equation}
where $Op^{W}$ denotes the Weyl-quantization, i.e.
\[
Op^{W}(a)f(u)=(2\pi)^{-n}\int_{\mathbb{R}^n}\int_{\mathbb{R}^n}e^{i(u-v)\xi}a\left(\xi,\frac{u+v}{2}\right)f(v)\,dv\,d\xi\,,
\] 
for $f\in \mathcal{S}(\mathbb{R}^n)$ and $u \in \mathbb{R}^n$, where $\pi_{\lambda}$ denotes the Schr\"{o}dinger representations of $\mathbb{H}_n$,  \\

   Going back to our case, of one keeps the same notation $\pi_{\lambda,\mu}$ for the infinitesimal representation, we compute that:
   \begin{flushleft}
   $\quad \pi_{\lambda,\mu}(X_1)=\partial_{u}=Op(i\xi)\,,$\\
    $\quad \pi_{\lambda,\mu}(X_2)=\frac{i}{2}\left(\lambda u^2-\frac{\mu}{\lambda}\right)=Op\left( \frac{i\lambda u^2}{2}-\frac{i\mu}{2 \lambda}\right)\,,$  \\
       $\quad\pi_{\lambda,\mu}(X_3)=-i\lambda u=Op(-i\lambda u)\,,$\\
       $\quad  \pi_{\lambda,\mu}(X_4)=i \lambda=Op(i \lambda)\,, $
   \end{flushleft}
   thus 
   \[
   \pi_{\lambda,\mu}(\mathcal{L})=\pi_{\lambda,\mu}(X_1)^2+\pi_{\lambda,\mu}(X_2)^2=\frac{
   d^2}{du^2}-\frac{1}{4}\left(\lambda u^2-\frac{
   \mu}{\lambda}\right)^2=-Op \left(\xi^2+\frac{1}{4}\left(\lambda u^2-\frac{\mu}{\lambda}\right)^2 \right)\,.
   \]
   With our choice of notation, the Plancherel measure of the Engel group $\mathcal{B}_4$ is $(2^{-3}\pi^{-4}) d\lambda\,d\mu$, in the sense that following expression for the Plancherel formula
   \begin{equation}\tag{4}\label{Planch.}
   \int_{\mathcal{B}_4} |f(x_1,x_2,x_3,x_4)|^2\,dx_1\,dx_2\,dx_3\,dx_4=2^{-3}\pi^{-4} \int_{\lambda \neq 0}\int_{\mu \in \mathbb{R}}\|\pi_{\lambda,\mu}(f)\|_{\textrm{HS}}^{2}\,d\mu\, d\lambda\,,
   \end{equation}
   holds for any $f \in \mathcal{S}(\mathbb{R})$, where $\|\cdot\|_{\textrm{HS}}$ denotes the Hilbert-Schmidt norm of an operator on $L^2$, that is $\|A\|_{\textrm{HS}}:=Tr(A^{*}A)$. The last allows for an extension of the group Fourier transform to $L^2(\mathcal{B}_4)$, and in particular formula \eqref{Planch.} holds true for any $f \in L^2(\mathcal{B}_4)$. \\
   
   Indeed, by using \eqref{g.f.t} the operator $\pi_{\lambda,\mu}(f)$ has integral kernel 
   \[
   \mathcal{K}_{f,\lambda,\mu}(u,v)=2\pi \int_{\mathbb{R}}e^{i(u-v)\xi} \mathcal{F}_{\mathbb{R}^4}(f)(\xi,\frac{\lambda}{2}v^2-\frac{\mu}{2\lambda},-\lambda v, \lambda)\,d\xi\,,
   \] 
   or equivalently
   \[
   \mathcal{K}_{f,\lambda,\mu}(u,v)=(2\pi)^{\frac{3}{2}}\mathcal{F}_{\mathbb{R}^3}(f)(v-u,\frac{\lambda}{2}v^2-\frac{\mu}{2\lambda},-\lambda v, \lambda)\,,
   \] 
   where the Fourier transform is taken with respect to the second, the third and the fourth variable of $f$. Integrating the $L^2(\mathbb{R}\times \mathbb{R})$-norm of $\mathcal{K}_{f,\lambda,\mu}$ (or the Hilbert-Schmidt norm of $\pi_{\lambda,\mu}(f)$) against $d\lambda,d\mu$ we obtain 
   \begin{eqnarray*}
\lefteqn{\int_{\mathbb{R}\setminus \{0\}} \int_{\mathbb{R}}\int_{\mathbb{R}} \int_{\mathbb{R}} \vert \mathcal{K}_{f,\lambda,\mu}(u,v) \vert ^2\, du\,dv\,d\mu\,d\lambda}\\
&=& (2\pi)^3 \int_{\mathbb{R}\setminus \{0\}} \int_{\mathbb{R}} \int_{\mathbb{R}} \int_{\mathbb{R}} \vert  \mathcal{F}_{\mathbb{R}^{3}}(f)(u-v,\frac{\lambda}{2}v^2-\frac{\mu}{2\lambda},-2\lambda,\lambda)\vert^2\,du\,dv\,d\lambda\,d\mu\\
&=& (2\pi)^3\int_{\mathbb{R}\setminus \{0\}} \int_{\mathbb{R}} \int_{\mathbb{R}} \int_{\mathbb{R}} \vert  \mathcal{F}_{\mathbb{R}^{3}}(f)(x_1,w_2,w_3,w_4)  \vert^2\, \frac{1}{2}\,dw_2\,dw_3\,dw_4\,dx_1\,,
\end{eqnarray*}
where the constant $\frac{1}{2}$ comes from the calculation of the determinant of the Jacobian matrix of the linear transformation $F(u,v,\lambda, \mu)=(w_1=u-v,w_2=\frac{\lambda}{2}v^2-\frac{\mu}{2\lambda},w_3=-\lambda v,w_4=\lambda)$. Finally, the Plancherel formula on $\mathbb{R}^3$ in the variable $(w_2,w_3,w_4)$ with dual variable $(x_2,x_3,x_4)$ gives 
\[
\int_{\mathbb{R}\setminus \{0\}} \int_{\mathbb{R}^3}|\mathcal{K}_{f,\lambda,\mu}(u,v)|^2\,dv\,du\,d\mu\,d\lambda=2^2 \pi^3 \int_{\mathbb{R}^4}|f(x_1,x_2,x_3,x_4)|^2\,dx_1\,dx_2\,dx_3\,dx_4\,, 
\] 
and the last implies \eqref{Planch.}.
\section{Difference operators}
\textit{Difference operators} on the setting of a compact Lie group introduced in \cite{RT10} as acting on Fourier coefficients, while on graded Lie groups in \cite{FR16}. In the setting of the Engel group $\mathcal{B}_4$ this yields the definition of the difference operators $\Delta_{x_i}$ as:
\[
\Delta_{x_i}\hat{\kappa}(\pi_{\lambda,\mu}):=\pi_{\lambda,\mu}(x_i \kappa)\,,\quad i=1,\cdots,4\,,
\]
for suitable distributions $\kappa$ on $\mathcal{B}_4$.\\

\indent To find the explicit expressions of the difference operators $\Delta_{x_i}$ we make use of the following property: For $X$ and $\tilde{X}$ being a left and a right invariant vector field, respectively, in the Lie algebra $\mathfrak{l}_4$, and for a distribution $\kappa$ on $\mathcal{B}_4$ we have 
\[
\pi_{\lambda,\mu}(X\kappa)=\pi_{\lambda,\mu}(X)\pi_{\lambda,\mu}(\kappa)\,,\quad \pi_{\lambda,\mu}(\tilde{X}\kappa)=\pi_{\lambda,\mu}(\kappa)\pi_{\lambda,\mu}(X)\,.
\] Notice that the right invariant vector fields that generate $\mathfrak{l}_4$ can be calculated as:
\[
\tilde{X}_1=\partial_{x_1}-x_2\partial_{x_3}-x_3\partial_{x_4}\,,\tilde{X}_2=\partial_{x_2}\,,\tilde{X}_3=\partial_{x_3}\,,\tilde{X}_4=\partial_{x_4}\,.
\]  
\begin{prop}\label{first.dif.op}
For suitable distribution $\kappa$ on $\mathcal{B}_4$ we have:
\[
\Delta_{x_1}\hat{\kappa}(\pi_{\lambda,\mu})=\frac{i}{\lambda}(\pi_{\lambda,\mu}(X_3)\pi_{\lambda,\mu}(\kappa)-\pi_{\lambda,\mu}(\kappa)\pi_{\lambda,\mu}(X_3))\,,
\]
where $\pi_{\lambda,\mu}(X_3)=-i\lambda u$, and 
\[
\Delta_{x_2} \hat{\kappa}(\pi_{\lambda,\mu})=\frac{2\lambda}{i}\partial_{\mu}\pi_{\lambda,\mu}(\kappa)\,.
\]
\end{prop}
\begin{proof}
Since $\pi_{\lambda,\mu}(X_4)=i \lambda$, and $\tilde{X}_3-X_3=X_4x_1$, we have 
\begin{align*}
\pi_{\lambda,\mu}(x_1 \kappa)&= \frac{1}{i \lambda}\pi_{\lambda,\mu}(X_4x_1 \kappa)=\frac{1}{i \lambda}((\tilde{X}_3-X_3)\kappa)\\
&=\frac{i}{\lambda}(\pi_{\lambda,\mu}(X_3)\pi_{\lambda,\mu}(\kappa)-\pi_{\lambda,\mu}(\kappa)\pi_{\lambda,\mu}(X_3))\,.
\end{align*}
Now, for the difference operator corresponding to $x_2$, we differentiate the group Fourier transform of $\kappa$ as in \eqref{fourier,tansform} at $h$ with respect to $\mu$ and get 
\begin{align*}
\partial_{\mu} \{ \pi_{\lambda, \mu}(\kappa)h(u)\}&=\partial_{\mu} \Big\{\int_{\mathbb{R}^4}\kappa(x)\exp\left(i\left(\frac{\mu}{2\lambda}x_2-\lambda x_4 \right)\right)\\
	& \cdot \exp\left(i\left(\lambda x_3 (u-x_1)-\frac{\lambda}{2}x_2 (u-x_1)^2\right)\right)h(u-x_1)dx \Big\}\\
	&=\int_{\mathbb{R}^4}\kappa(x)\exp\left(i\left(\frac{\mu}{2\lambda}x_2-\lambda x_4 \right)\right)\\
	& \cdot \exp\left(i\left(\lambda x_3 (u-x_1)-\frac{\lambda}{2}x_2 (u-x_1)^2\right)\right)h(u-x_1)\left( \frac{i}{2\lambda} x_2 \right)dx\,,
\end{align*}
or in terms of difference operators,
		\[
			\partial_{\mu} \pi_{\lambda, \mu}(\kappa)=\pi_{\lambda, \mu} \left(\frac{i}{2\lambda}x_2 \kappa \right)=\frac{i}{2 \lambda} \Delta_{x_2}\pi_{\lambda, \mu}(\kappa)\,.
		\]
\end{proof}
\begin{prop}\label{sec.dif.op}
For a suitable distribution $\kappa$ we have:
\[
\Delta_{x_3}\hat{\kappa}(\pi_{\lambda,\mu})=\frac{i}{\lambda}(\Delta_{x_2}\pi_{\lambda,\mu}(\kappa) \pi_{\lambda,\mu}(X_3)+\pi_{\lambda,\mu}(\kappa)\pi_{\lambda,\mu}(X_1)-\pi_{\lambda,\mu}(X_1)\pi_{\lambda,\mu}(\kappa))\,,
\]
where $\Delta_{x_2|\pi_{\lambda,\mu}}$ is given in Proposition \ref{first.dif.op} and $\pi_{\lambda,\mu}(X_1)=\partial_{u}$, $\pi_{\lambda,\mu}(X_3)=-i\lambda u$.
\end{prop}
\begin{proof}
Since $X_1-\tilde{X}_1-x_2X_3= \partial_{x_4}x_3$ we have 
\begin{align*}
\pi_{\lambda,\mu}(x_3 \kappa)&=\frac{1}{i \lambda}(X_4x_3 \kappa)=\frac{1}{i \lambda}((X_1-\tilde{X}_1-x_2\tilde{X}_3)\kappa)\\
&=\frac{1}{i \lambda}(\pi_{\lambda,\mu}(X_1)\pi_{\lambda,\mu}(\kappa)-\pi_{\lambda,\mu}(\kappa)\pi_{\lambda,\mu}(X_1)-\Delta_{x_2}\pi_{\lambda,\mu}(\kappa)\pi_{\lambda,\mu}(X_3)\\
&= \frac{i}{\lambda}(\Delta_{x_2}\pi_{\lambda,\mu}(\kappa)\pi_{\lambda,\mu}(X_3)+\pi_{\lambda,\mu}(\kappa)\pi_{\lambda,\mu}(X_1)-\pi_{\lambda,\mu}(X_1)\pi_{\lambda,\mu}(\kappa))\,,
\end{align*}
completing the proof. 
\end{proof}
\begin{prop}
For a suitable distribution $\kappa$ on $\mathcal{B}_4$ we have:
	\begin{align*}
	(\Delta_{x_4} \pi_{\lambda, \mu}(\kappa))h(u)&= i\partial_{\lambda}\{\pi_{\lambda, \mu}(\kappa)h(u)\} -\left(\frac{\mu}{2 \lambda^2}+\frac{u^2}{2} \right)\Big \{\Delta_{x_2}\pi_{\lambda, \mu}(\kappa)h(u) \Big\}\\
	&+u\Big \{ \Delta_{x_3}\pi_{\lambda, \mu}(\kappa)h(u)\Big \}-\Big \{ \Delta_{x_3}\Delta_{x_1}\pi_{\lambda, \mu}(\kappa)h(u)\Big\}\\
	&+u\Big \{ \Delta_{x_2}\Delta_{x_1}\pi_{\lambda, \mu}(\kappa)h(u) \Big \}-\frac{1}{2}\Big\{\Delta_{x_2}\Delta_{x_1}^{2}\pi_{\lambda, \mu}(\kappa)h(u)\Big\}\,,
	\end{align*}
	where the difference operators $\Delta_{x_i|\pi_{\lambda, \mu}}$, $i=1,2,3$, are given in Propositions \ref{first.dif.op} and \ref{sec.dif.op}, respectively.
\end{prop}
\begin{proof}
Differentiating the group Fourier transform of $\kappa$ as in \eqref{fourier,tansform} at $h$ with respect to $\lambda$ yields
\begin{align*}
	\partial_{\lambda}\{\pi_{\lambda, \mu}(\kappa)h(u)\}&= \partial_{\lambda} \Big \{ \int_{\mathbb{R}^4} \kappa(x) \exp \left( i \left( \frac{\mu}{2 \lambda}x_2 -\lambda x_4\right) \right) \\
	&\cdot \exp \left(i \left(\lambda x_3(u-x_1)-\frac{\lambda}{2}x_2(u-x_1)^2 \right) \right)h(u-x_1)\,dx \Big\}\\
	&=\int_{\mathbb{R}^4}\kappa(x) \exp \left( i \left( \frac{\mu}{2 \lambda}x_2 -\lambda x_4+\lambda x_3(u-x_1)-\frac{\lambda}{2}x_2(u-x_1)^2 \right) \right)\\
	& h(u-x_1) \Big\{i\left(-\frac{\mu}{2 \lambda^2}x_2-x_4+x_3(u-x_1)-\frac{x_2}{2}(u-x_1)^2 \right) \Big\} dx\,.
	 	\end{align*}
	Rewriting the above formula in terms of difference operators we obtain
	\begin{align*}
	\partial_{\lambda}\{\pi_{\lambda, \mu}(\kappa)h(u)\}&= i \Big[ -\left(\frac{\mu}{2 \lambda^2}+\frac{u^2}{2} \right)\Big \{\Delta_{x_2}\pi_{\lambda, \mu}(\kappa)h(u) \Big\}-\Big \{\Delta_{x_4}\pi_{\lambda, \mu}(\kappa)h(u)\Big\}\\
	&+u\Big \{ \Delta_{x_3}\pi_{\lambda, \mu}(\kappa)h(u)\Big \}-\Big \{ \Delta_{x_3}\Delta_{x_1}\pi_{\lambda, \mu}(\kappa)h(u)\Big\}\\
	&+u\Big \{ \Delta_{x_2}\Delta_{x_1}\pi_{\lambda, \mu}(\kappa)h(u) \Big \}-\frac{1}{2}\Big\{\Delta_{x_2}\Delta_{x_1}^{2}\pi_{\lambda, \mu}(\kappa)h(u)\Big\}\Big]\,,
	\end{align*}
completing the proof.	

\end{proof}
\noindent For example we have:
\begin{flushleft}
$  \quad \Delta_{x_1}\pi_{\lambda,\mu}(X_1)=-I\,, \Delta_{x_1}\pi_{\lambda,\mu}(X_2)=\Delta_{x_1}\pi_{\lambda,\mu}(X_3)=\Delta_{x_1}\pi_{\lambda,\mu}(X_4)=0$\\
$  \quad \Delta_{x_2}\pi_{\lambda,\mu}(X_1)=\Delta_{x_2}\pi_{\lambda,\mu}(X_3)=\Delta_{x_2}\pi_{\lambda,\mu}(X_4)=0\,, \Delta_{x_2}\pi_{\lambda,\mu}(X_2)=-\lambda I$\\
$ \quad \Delta_{x_3}\pi_{\lambda,\mu}(X_1)=\Delta_{x_3}\pi_{\lambda,\mu}(X_4)=0\,,\Delta_{x_3}\pi_{\lambda,\mu}(X_2)=-\lambda u +u\,, \Delta_{x_3}\pi_{\lambda,\mu}(X_3)=-I$\\
$  \quad \Delta_{x_4}\pi_{\lambda,\mu}(X_1)=\Delta_{x_4}\pi_{\lambda,\mu}(X_4)=0\,, \Delta_{x_4}\pi_{\lambda,\mu}(X_2)=\frac{u^2}{2}(1-\lambda)+\frac{\mu}{2\lambda}\,, \Delta_{x_4}\pi_{\lambda,\mu}(X_4)=-I$,
\end{flushleft}  
where the difference operators $\Delta_{x_i}\pi_{\lambda,\mu}(X_j)$ can be understood as the group Fourier transform of the distribution $x_i X_j \delta_0$.
\section{Quantization and symbol classes}
In this note, we may slightly change the notation of the symbol introduced in \cite{FR16}. We keep the notation
\[
x=(x_1,x_2,x_3,x_4) \in \mathcal{B}_4\,,
\] to denote the coordinates of an element in the Engel group $\mathcal{B}_4$, and we may denote by 
\[
\sigma(x,\lambda,\mu):=\sigma(x,\pi_{\lambda,\mu})\,,\quad (x, \lambda, \mu) \in \mathcal{B}_4 \times \mathbb{R}\setminus \{0\} \times \mathbb{R}\,,
\]
the symbol $\sigma$ parametrised by $(x, \lambda,\mu)$.  In addition, if the multi-index $\alpha \in \mathbb{N}^{4}_{0}$ is written as
\[
\alpha=(\alpha_1,\alpha_2,\alpha_3,\alpha_4)\,,\quad \alpha_i \in \mathbb{N}_{0}\,,
\] 
then the homogeneous degree of $\alpha$ is given by:
\[
[\alpha]=\alpha_1+\alpha_2+2\alpha_3+3\alpha_4\,.
\]
For each $\alpha$ we may write:
\[
x^{\alpha}=x_1^{\alpha_1}x_2^{\alpha_2}x_3^{\alpha_3}x_4^{\alpha_4}\,,
\] 
so that the corresponding difference operator can be defined as:
\[
{\Delta^{'}}^{\alpha}=\Delta_{x_1}^{\alpha_1}\Delta_{x_2}^{\alpha_2} \Delta_{x_3}^{\alpha_3}\Delta_{x_4}^{\alpha_4}\,.
\]
Finally for the vector field $X$ we write $X^{\alpha}$ to denote the following composition of vector fields:
\[
X_{1}^{\alpha_1}X_{2}^{\alpha_2}X_{3}^{\alpha_3}X_{4}^{\alpha_4}\,.
\]
Following \cite{FR16} we define the symbol classes $S_{\rho,\delta}^{m}(\mathcal{B}_4)$, where $0 \leq \delta <\rho \leq 1$ and $m \in \mathbb{R}$, as the set of symbols $\sigma$ for which the following quantities are finite:
\[
\|\sigma\|_{S^{m}_{\rho,\delta},a,b,c}:= \sup_{\lambda \in \mathbb{R}\setminus \{0\}, \mu \in \mathbb{R},x \in \mathcal{B}_4}\|\sigma(x,\lambda,\mu)\|_{S^{m}_{\rho,\delta},a,b,c}\,,\quad a,b,c \in \mathbb{N}_{0}\,,
\]
where 
\[
\|\sigma(x,\lambda,\mu)\|_{S^{m}_{\rho,\delta},a,b,c}:= \sup_{\substack{[a] \leq a\\ [\beta] \leq b,|\gamma| \leq c}} \|\pi_{\lambda,\mu}(I-\mathcal{L})^{\frac{\rho [\alpha]-m-\delta[\beta]+\gamma}{2}}X^{\beta}{\Delta^{'}}^{\alpha}\sigma(x,\lambda,\mu)\pi_{\lambda,\mu}(I-\mathcal{L})^{-\frac{\gamma}{2}}\|_{op}\,.
\]

\indent There is a natural quantization on any type-I Lie group introduced by \cite{Tayl84} that can be served as the analogue of the Kohn-Nirenberg quantization on $\mathbb{R}^n$. In particular, the \textit{quantization}, i.e., the mapping $\sigma\mapsto Op(\sigma)$ produces operators associated with a symbol $\sigma$ (for example in the class of symbols $S_{\rho,\delta}^{m}(\mathcal{B}_4)$) on $\mathcal{S}(\mathcal{B}_4)$ given by:
\begin{equation}\tag{5}\label{psido}
Op(\sigma)\phi(x)=2^{-3}\pi^{-4}\int_{\lambda \neq 0}\int_{\mu \in \mathbb{R}} Tr \left(\pi_{\lambda,\mu}(x)\sigma(x,\lambda,\mu)\pi_{\lambda,\mu}(\phi)\right)\,d\mu\,d\lambda\,.
\end{equation}
Here we have used our notation for the description of the dual, as well as for the symbol and the Plancherel measure, see \eqref{Planch.}.\\

Let us note that by \eqref{g.f.t}, we see that for the symbol $\sigma$ quantized as:
\[
\sigma(x,\lambda,\mu)=Op(a_{\kappa_x,\lambda,\mu})\,,
\]
then its symbol that is given by
\[
a_{\kappa_x,\lambda,\mu}(v,\xi)=(2\pi)^2\mathcal{F}_{\mathbb{R}^4}(\kappa_x)(\xi,\frac{\lambda}{2}v^2-\frac{\mu}{2\lambda},-\lambda v, \lambda)\,,
\]
shall be called the $(\lambda,\mu)$-symbol,
where $\{\kappa_{x}(y)\}$ is the kernel of the symbol $\{\sigma(x,\lambda,\mu)\}$, i.e.,
\[
\sigma(x,\lambda,\mu)=\pi_{\lambda,\mu}(\kappa_x)\,.
\]
The above, together with the property of the Fourier transform 
\[
\hat{\phi}(\pi_{\lambda,\mu})\pi_{\lambda,\mu}(x)=\mathcal{F}_{\mathcal{B}_4}(\phi(x\cdot))(\pi_{\lambda,\mu})\,,
\]
 and the properties of the trace yield the following alternative formula for the quantization given in \eqref{psido}:
 \[
 Op(\sigma)(\phi)(x)=2^{-3}\pi^{-4}\int_{\lambda \neq 0}\int_{\mu \in \mathbb{R}} Tr \left( Op(a_{\kappa_x,\lambda,\mu})Op(a_{\phi(x\cdot),\lambda,\mu}) \right)\,d\mu\,d\lambda\,.
 \]
The last formula shows that the quantization formula \eqref{psido} can be expressed in terms of composition of quantization of symbols in the Euclidean space.\\

 Similarly, for the case of the Heisenberg group $\mathbb{H}_n$, \eqref{gfthei} implies that the operator $Op(\sigma)$ on $\mathcal{S}(\mathbb{H}_n)$ involves 'Euclidean objects', and in particular:
\[
Op(\sigma)\phi(x)=c_n \int_{\lambda \neq 0} Tr \left(Op^{W}(a_{x,\lambda})Op^{W} [\mathcal{F}_{\mathbb{R}^{2n+1}}(\phi(x\cdot))(\sqrt{|\lambda|}\cdot,\sqrt{\lambda}\cdot,\lambda)] \right)|\lambda|^{n}\,d\lambda\,,
\]
where the symbol $a_{x,\lambda}$ (also called the $\lambda$-symbol) given by:
\[
a_{x,\lambda}(\xi,u)=\mathcal{F}_{\mathbb{R}^{2n+1}}(\kappa^{'}_{x})(\sqrt{|\lambda|}\xi,\sqrt{\lambda}u,\lambda)\,,
\]
where $\{\kappa^{'}_{x}(y)\}$ is the kernel of the symbol $\sigma$,
 is such that 
\[
\sigma(x,\lambda):=\sigma(x,\pi_{\lambda})=Op^{W}(a_{x,\lambda})\,.
\]
For our notation, especially for the Plancherel measure $c_n|\lambda|^n$ on $\mathbb{H}_n$, see \cite[Chapter 6]{FR16}.\\

In contrast with the case of the Engel group $\mathcal{B}_4$, in the setting of the Heisenberg group $\mathbb{H}_n$, one can renormalise $a_{x,\lambda}$ as
\[
a_{x,\lambda}(\xi,u):=\tilde{a}_{x,\lambda}(\sqrt{|\lambda|}\xi,\sqrt{\lambda}u)\,,
\]
and therefore, one can characterise the symbol classes $S_{\rho,\delta}^{m}(\mathbb{H}_n)$ by the property that these $\lambda$-symbols belong to some Shubin spaces, called $\lambda$-type version of the usual Shubin classes, leading to sufficient criteria for ellipticity and hypoellipticity  of operators on $\mathbb{H}_n$ in terms of the invertibility properties of their $\lambda$-symbols, see \cite[Chapter 6]{FR16}.

\section*{Acknowledgement}
I would like to thank Professor Michael Ruzhansky for introducing me to this topic and for comments leading to improvements of the current work.

\end{document}